\documentclass[reqno]{amsart}

\usepackage{amsmath}
\usepackage{amssymb}

\numberwithin{equation}{section}

\newtheorem{Theorem}{Theorem}[section]
\newtheorem{Lemma}[Theorem]{Lemma}

\theoremstyle{definition}
\newtheorem{Definition}[Theorem]{Definition}
\newtheorem{Remark}[Theorem]{Remark}

\newcommand{\thref}[1]{Theorem \ref{#1}}
\newcommand{\leref}[1]{Lemma \ref{#1}}

\newcommand{\reref}[1]{Remark \ref{#1}}
\newcommand{\deref}[1]{Definition \ref{#1}}

\newcommand{\seref}[1]{Section \ref{#1}}
\newcommand{\ssref}[1]{Subsection \ref{#1}}

\newcommand{\dx}{{\frac{d}{dx}}}
\newcommand{\dxx}{{\frac{d^2}{dx^2}}}

\newcommand{\Cset}{\mathbb{C}}

\newcommand{\Zset}{\mathbb{Z}}
\newcommand{\Nset}{\mathbb{N}}

\newcommand{\Span}{\mathrm{span}}
\newcommand{\Wr}{\mathrm{Wr}}

\newcommand{\al}{{\alpha}}
\newcommand{\be}{{\beta}}
\newcommand{\ga}{{\gamma}}
\newcommand{\la}{{\lambda}}

\newcommand{\de}{{\delta}}

\newcommand{\Ga}{{\Gamma}}

\newcommand{\ep}{{\epsilon}}

\newcommand{\dms}{d\mu_{\mathrm{d}}(s)}

\newcommand{\fh}{\hat{f}}
\newcommand{\ah}{\hat{a}}
\newcommand{\bh}{\hat{b}}
\newcommand{\ch}{\hat{c}}

\newcommand{\cA}{{\mathcal A}}
\newcommand{\cAb}{\bar{{\mathcal A}}}
\newcommand{\cD}{{\mathcal D}}
\newcommand{\cDb}{\bar{{\mathcal D}}}

\newcommand{\cP}{{\mathcal P}}
\newcommand{\cQ}{{\mathcal Q}}
\newcommand{\cR}{{\mathcal R}}

\newcommand{\cM}{{\mathcal M}}
\newcommand{\cJ}{{\mathcal J}}
\newcommand{\cJa}{{\mathcal J}^{(\alpha)}}

\newcommand{\La}{L^{\alpha}}
\newcommand{\Lh}{\hat{L}}
\newcommand{\Lha}{\hat{L}^{\alpha;\beta}}
\newcommand{\Ja}{J^{(\alpha)}}
\newcommand{\Jh}{\hat{J}}
\newcommand{\cJh}{\hat{{\mathcal J}}}
\newcommand{\Bh}{\hat{B}}

\newcommand{\Bb}{\bar{B}}
\newcommand{\Bt}{\tilde{B}}

\newcommand{\fR}{{\mathfrak R}}

\newcommand{\cce}{c^{\varepsilon}}
\newcommand{\bbe}{b^{\varepsilon}}
\newcommand{\aae}{a^{\varepsilon}}

\newcommand{\fD}{\mathfrak{D}}
		
\begin{document}
\title[Krall-Laguerre commutative algebras of differential operators]{Krall-Laguerre commutative algebras of ordinary differential operators}

\date{July 10, 2011}

\author[P.~Iliev]{Plamen~Iliev}
\address{School of Mathematics, Georgia Institute of Technology,
Atlanta, GA 30332--0160, USA}
\email{iliev@math.gatech.edu}
\thanks{The author is supported in part by NSF Grant \#0901092}

\begin{abstract}
In 1999, Gr\"unbaum, Haine and Horozov defined a large family of commutative algebras of ordinary differential operators which have orthogonal  polynomials as eigenfunctions. These polynomials are mutually orthogonal with respect to a Laguerre-type weight distribution, thus providing solutions to Krall's problem. In the present paper we give a new proof of their result which establishes a conjecture, concerning the explicit characterization of the dual commutative algebra of eigenvalues. In particular, for the Koornwinder's generalization of Laguerre polynomials, our approach yields an explicit set of generators for the whole algebra of differential operators. We also illustrate how more general Sobolev-type orthogonal polynomials fit within this theory. 
\end{abstract}

\maketitle

\tableofcontents

\section{Introduction}\label{se1}
In 1929, Bochner \cite{Bo} characterized the classical orthogonal polynomials (Jacobi, Laguerre, Hermite and Bessel) as the only ones that satisfy a differential equation of the form
\begin{equation}\label{1.1}
B\left(x,\dx\right)p_n(x)=\lambda_n p_n(x),\quad n\in\Nset_0,
\end{equation}
where $B(x,\dx)=\sum_{j=0}^{2}b_j(x)\frac{d^j}{dx^j}$ is a second-order differential operator with coefficients $b_j(x)$ independent of the degree index $n$ and with eigenvalue $\lambda_n$ independent of the variable $x$.
In 1938, H.~L.~Krall \cite{Kr1} posed the general problem to find all families of orthogonal polynomials $\{p_n(x)\}_{n=0}^{\infty}$ which are eigenfunctions of a differential operator $B=\sum_{j=0}^{m}b_j(x)\frac{d^j}{dx^j}$ of arbitrary order. In particular, he showed that the order $m$ must be even and two years later \cite{Kr2} he completed the classification for operators of order $4$, by adding three new families of orthogonal polynomials. During the last 70 years numerous families of orthogonal polynomials which are eigenfunctions of higher-order differential operators $B$ were constructed. 

Inspired by the bispectral problem \cite{DG}, Gr\"unbaum and Haine discovered in \cite{GH} that the new families introduced by Krall can be obtained from the classical ones, by applying the Darboux transformation to the Jacobi recurrence operator. This brought new ``bispectral" techniques in this old problem. Up to a linear change of the variable $x$, all known solutions of Krall's problem belong to one of the families below:
\begin{itemize}
\item[{[i]}] The Laguerre-type polynomials introduced by Gr\"unbaum, Haine and Horozov \cite{GHH} by applying the Darboux transformation to the recurrence operator for the Laguerre polynomials. These polynomials are orthogonal on $[0,\infty)$ with respect to a weight distribution of the form $$w(x)=\frac{1}{\Gamma(\al+1)}x^{\al}e^{-x}+\sum_{\text{finitely many }j\in\Nset_0}(-1)^ju_j\de^{(j)}(x), \text{ where }\al\in\Nset_0.$$
\item[{[ii]}] The Jacobi-type polynomials introduced by Gr\"unbaum and Yakimov \cite{GY} by applying the Darboux transformation to the recurrence operator for the Jacobi polynomials. These polynomials are orthogonal on $[-1,1]$ with respect to one of the following weight distributions:
\begin{align*}
w_1(x)&=\frac{\Gamma(\al+\be+2)}{2^{\al+\be+1}\Gamma(\al+1)\Gamma(\be+1)}(1-x)^{\al}(1+x)^{\beta}\\
&\qquad+\sum_{\text{finitely many }j\in\Nset_0}(-1)^ju_j\de^{(j)}(x-1), \text{ where }\al\in\Nset_0, \quad \beta>-1,
\intertext{or}
w_2(x)&=\frac{\Gamma(\al+\be+2)}{2^{\al+\be+1}\Gamma(\al+1)\Gamma(\be+1)}(1-x)^{\al}(1+x)^{\beta}\\
&\qquad+\sum_{\text{finitely many }j\in\Nset_0}(-1)^ju_j\de^{(j)}(x-1)\\
&\qquad+\sum_{\text{finitely many }j\in\Nset_0}(-1)^jv_j\de^{(j)}(x+1), \text{ where }\al,\be\in\Nset_0.
\end{align*}
\end{itemize}
Both papers use the general bispectral techniques developed in \cite{BHY}. Moreover, the results in \cite{GHH,GY} imply that to each of the families above, we can associate a ``nontrivial'' commutative algebra $\cDb$ of differential operators $B$ such that \eqref{1.1} holds for some $\la_n$. By ``nontrivial'', we mean that the operators in the algebra are not obtained by taking constant coefficient polynomials of a fixed operator.  The proofs in \cite{GHH,GY} guarantee the existence of operators in the algebra $\cDb$ for every sufficiently large even order, thus raising the natural question of giving an explicit characterization of this algebra (and in particular, the operator of minimal order).

Recently \cite{I2}, we studied the algebra of differential operators for the Jacobi-type polynomials and used the results to obtain multivariate extensions of the family [ii] within the context of partial differential operators invariant under rotations. The main goal of the present paper is to explain how these techniques can be adapted for the Laguerre-type polynomials and to establish a conjecture in \cite{GHH} concerning the explicit characterization of the isomorphic algebra $\cAb$ of eigenvalues $\la_n$. We also illustrate how our approach leads to explicit formulas for the differential operators and how more general Sobolev-type orthogonal polynomials fit within the above theory. This leads to a uniform description of a large family of commutative algebras of ordinary differential operators which have Laguerre-type orthogonal polynomials as eigenfunctions, containing all known examples.

A more detailed description of the results in the paper is as follows. Using the Darboux transformation, one can naturally associate to each of the Laguerre-type weights listed in [i] a certain Casorati determinant $\tau(n)\in\Cset[n]$. We denote by $\cA$ the algebra of all $\la_n\in\Cset[n]$ such that $\la_{n}-\la_{n-1}$ is divisible by $\tau(n-1)$ and we prove that $\cA\subset\cAb$ (i.e. for every $\la_n\in\cA$ there exists a differential operator $B$ satisfying \eqref{1.1}). Next, we show that if the parameters $u_j$ do not satisfy a certain algebraic relation, then $\cA=\cAb$. In other words for generic parameters $\{u_j\}$, all possible eigenvalues $\la_n$ in equation \eqref{1.1} are in $\cA$. This fact was observed in several examples in \cite{GHH}. A similar rule applies (generically) in the purely continuous version of the bispectral problem (with the discrete derivative $\la_{n}-\la_{n-1}$ replaced by the usual one), see \cite{DG}. However, our methods and arguments are very different from the ones used in the pioneering work of Duistermaat and Gr\"unbaum. Although we focus mainly on the Laguerre-type polynomials [i], the results apply directly to more general Laguerre-type polynomials orthogonal with respect to a Sobolev-type inner product, which we discuss briefly in the examples. 

The paper is organized as follows. In the next section we introduce the Laguerre polynomials and the basic notations which we use in rest of the paper. In \seref{se3} we briefly review the sequence of Darboux transformations from the recurrence operator of Laguerre polynomials which leads to the families listed in [i] and the construction of $\tau(n)$. We introduce certain free parameters $\{\be_j\}$ which parametrize the Darboux transformations and which are in one to one correspondence with the parameters $\{u_j\}$ appearing in the weight distribution. In \seref{se4} we prove that $\cA\subset\cAb$. In \seref{se5} we show that for generic parameters $\{\be_j\}$ (or equivalently, for generic $\{u_j\}$) we have $\cA=\cAb$. At the end of the paper we list three instructive examples. In the first example, we consider the Koornwinder's generalized Laguerre polynomials \cite{Ko} which correspond to weights of the form 
$$w(x)=\frac{1}{\Gamma(\al+1)}x^{\al}e^{-x}+u_0\de(x), \text{ where }\al\in\Nset_0.$$
The computation of the differential operator of minimal order has a long history and goes back to the works of H.~L.~Krall, A.~M.~Krall and L.~L.~Littlejohn \cite{Kr,KL,Kr1,Kr2,L}. The explicit form of this operator for arbitrary $\al\in\Nset$ was obtained by J.~Koekoek and R.~Koekoek \cite{KK}. Our techniques applied to this case show that $\cA=\cAb$ and yield an explicit set of generators for the whole algebra $\cDb$. In the second example, we consider in a detail the case when 
$$w(x)=\frac{1}{\Gamma(\al+1)}x^{\al}e^{-x}+u_0\de(x)-u_1\de'(x), \text{ where }\al\in\Nset_0.$$
We write an explicit formula for the algebraic equation which allows to determine whether $u_0$ and $u_1$ are generic or not. In particular, when $w(x)=e^{-x}+u_0\de(x)-u_1\de'(x)$ we show that $\cA$ is a proper subalgebra of $\cAb$ when $u_0$ and $u_1$ are not generic. In the third example, we explain how our results apply to Laguerre-type polynomials orthogonal with respect to the Sobolev inner product
\begin{equation*}
\langle f(x), g(x)\rangle=\frac{1}{\Gamma(\al+1)}\int_0^{\infty}f(x)g(x)x^{\al}e^{-x}dx+[f(0),f'(0)]\,A\,\left[\begin{matrix} g(0)\\ g'(0)\end{matrix}\right],
\end{equation*}
where $\al\in\Nset_0$ and $A$ is a symmetric $2\times 2$ matrix.
\section{Laguerre polynomials} \label{se2}
In this section we review briefly the basic properties of Laguerre polynomials. For $n\in\Nset_0$ and $\al>-1$ the Laguerre 
polynomials $\{\La_n(x)\}_{n=0}^{\infty}$ are defined by 
\begin{equation}\label{2.1}
\La_n(x)=\frac{\Ga(\al+n+1)}{\Ga(\al+1)\Ga(n+1)}\,{}_1F_1(-n;\al+1;x).
\end{equation}
Here and in the rest of the paper, we use the standard gamma function $\Ga(x)$ and ${}_1F_1$ denotes the confluent hypergeometric function of the first kind 
$${}_1F_1(a;c;x)=\sum_{k=0}^{\infty}\frac{(a)_k x^k}{(c)_k k!}.$$
It will be convenient to set $\La_{-j}(x)=0$ for every $j\in\Nset$.
The polynomials $\La_n(x)$ are mutually orthogonal on $[0,\infty)$ with respect to the gamma distribution $x^{\al}e^{-x}dx$. Explicitly, we have
\begin{equation*} 
\int_{0}^{\infty}\La_n(x)\La_m(x)x^{\al}e^{-x}dx=\frac{\Ga(\al+n+1)}{n!}\,\de_{n,m}.
\end{equation*}
What is important for us is that the Laguerre polynomials are simultaneously eigenfunctions of a second-order difference operator in the discrete variable $n$, which is independent of $x$, and of a differential operator in the continuous variable $x$, which is independent of $n$. More precisely, they solve the discrete-continuous bispectral problem
\begin{subequations}\label{2.2}
\begin{align}
-(n+1)\La_{n+1}(x)+(2n+\al+1)\La_n(x)-(n+\al)\La_{n-1}(x)&=x\La_{n}(x)\label{2.2a}\\
-x\frac{d^2 \La_n(x)}{dx^2}-(\al+1-x)\frac{d \La_n(x)}{dx}&=n\La_n(x).\label{2.2b}
\end{align}
\end{subequations}
From the explicit formula \eqref{2.1} one can easily deduce the following differential-difference relation
\begin{equation}\label{2.3}
\frac{d}{dx}\left[\La_n(x)-\La_{n-1}(x)\right]=-\La_{n-1}(x).
\end{equation}

By abuse of notation, we shall identify a function $f_n$ defined on $\Nset_0$ with the corresponding semi-infinite vector $[f_0,f_1,f_2,\dots]^t$. In particular, if we think of $\La_n(x)$ as the semi-infinite vector 
$$[\La_0(x),\La_1(x),\La_2(x),\dots]^{t}$$ 
then the recurrence relations \eqref{2.2a} can be written in the compact form 
\begin{equation}\label{2.4}
\Ja\La_n(x)=x\La_n(x),
\end{equation}
where $\Ja$ is the tridiagonal semi-infinite (Jacobi) matrix
\begin{subequations}\label{2.5}
\begin{equation}\label{2.5a}
\Ja=\left[\begin{matrix}
b_0&a_0\\
c_1&b_1&a_1\\
   &c_2&b_2&a_2\\
   &   &\ddots&\ddots&\ddots
\end{matrix}\right],
\end{equation}
with entries 
\begin{equation}\label{2.5b}
a_n=-(n+1), \quad b_n=2n+\al+1 , \quad c_n=-(n+\al).
\end{equation}
\end{subequations}

\section{Darboux transformations}\label{se3}
In this section, following \cite{GHH}, we describe the result of $k$ successive Darboux transformations starting from $\Ja$ for $\al\in\Nset$ and 
$k\leq\al$.

First we review briefly the lattice version of the Darboux transformation for bi-infinite matrices. 
Let $\cJ_0$ be a tridiagonal bi-infinite matrix with nonzero off-diagonal entries. If we factor $\cJ_0$ as a product of an upper-triangular and a lower-triangular matrices, then we can  produce a new matrix by exchanging the factors. Following  \cite{MS}, we shall refer to this operation as the Darboux transformation. If we iterate this process $k$ times, we obtain a new tridiagonal bi-infinite matrix $\cJh$ as follows
\begin{align}
&\cJ_0=\cP_0\cQ_0\curvearrowright \cJ_1=\cQ_0\cP_0=\cP_1\cQ_1
\curvearrowright\cdots                      \nonumber\\
&\qquad \cJ_{k-1}=\cQ_{k-2}\cP_{k-2}=\cP_{k-1}\cQ_{k-1}
\curvearrowright \cJh=\cJ_k=\cQ_{k-1}\cP_{k-1}.   \label{3.1}
\end{align}
From \eqref{3.1} it follows that 
\begin{equation}\label{3.2}
\cJh \cQ=\cQ\cJ_0
\end{equation}
and 
\begin{equation}\label{3.3}
\cJ_0^k=\cP\cQ, 
\end{equation}
where $\cP=\cP_0\cP_1\cdots \cP_{k-1}$ and $\cQ=\cQ_{k-1}\cQ_{k-2}\cdots \cQ_{0}$. Similarly to the semi-infinite case, we shall identify a function $f_n$ defined for $n\in\Zset$ with the corresponding bi-infinite vector $[\dots, f_{-2},f_{-1},f_0,f_1,f_{2},\dots]^t$. Thus we can think of the bi-infinite matrices as linear transformations acting on functions $f_n$ by matrix multiplication. 
Equations \eqref{3.2} and \eqref{3.3} imply that 
\begin{equation}\label{3.4}
\ker(\cQ)\subset\ker(\cJ_0^k)\text{ and }\cJ_0(\ker(\cQ))\subset\ker(\cQ).
\end{equation}
Conversely, one can show that if \eqref{3.4} holds then there exists a tridiagonal matrix $\cJh$ which 
is obtained by a sequence of Darboux transformations as in \eqref{3.1}.

The lower-triangular matrix $\cQ$ is uniquely determined by its kernel, up to a multiplication by a diagonal matrix on the left. More precisely, if $\{\psi^{(j)}_n\}_{j=0}^{k-1}$ is a basis for 
$\ker(\cQ)$ and $f_n$ is an arbitrary function then 
\begin{equation}\label{3.5}
\cQ f_n= g_n\Wr_n(\psi^{(0)}_n,\psi^{(1)}_n,\dots,\psi^{(k-1)}_n,f_n),
\end{equation}
for an appropriate function $g_n$. We use $\Wr_n$ to denote the discrete Wronskian (or Casorati determinant):
$$\Wr_n(g^{(1)}_n,g^{(2)}_n,\dots,g^{(k)}_n)=\det(g^{(i)}_{n-j+1})_{1\leq i,j\leq k}.$$
Combining the above remarks, we see that the sequence of Darboux 
transformations \eqref{3.1} for $\cJ_0$ is characterized by choosing a 
basis for $\ker(\cQ)$ satisfying
\begin{equation}\label{3.6}
\cJ_0\psi^{(0)}_n=0\text{ and }\cJ_0\psi^{(j)}_n=\psi^{(j-1)}_n
\text{ for }j=1,\dots,k-1.
\end{equation}

We want to apply the above construction with $\cJ_0=\cJa$ where $\cJa$ is an appropriate bi-infinite extension of 
the semi-infinite Jacobi matrix $\Ja$ for Laguerre polynomials. We define
\begin{equation*}
\cJa=\left[\begin{matrix}
\ddots&\ddots&\ddots\\
&\cce_{-1}&\bbe_{-1}&\aae_{-1}&\\
 &&\cce_0&\bbe_0&\aae_0\\
&&&\cce_1&\bbe_1&\aae_1\\
  & &&   &\ddots&\ddots&\ddots
\end{matrix}\right],
\end{equation*} 
where 
$$\aae_n=a_{n+\varepsilon}=-(n+\varepsilon+1), \quad \bbe_n=b_{n+\varepsilon}=2(n+\varepsilon)+\al+1, \quad \cce_n=c_{n+\varepsilon}=-(n+\varepsilon+\al).$$
Since $\al\in\Nset$ we see that for $\varepsilon\notin \Zset$, the off-diagonal entries $\aae_n$ and $\cce_n$ are nonzero. 

For $j\in\{0,1,\dots,k-1\}$ we define
\begin{equation}\label{3.7}
\phi^{1,j}_n=\frac{(-1)^j}{(1-\al)_j}\binom{n+j}{j}\text{ and }
\phi^{2,j}_n=\frac{(-1)^j}{j!}\binom{n+\al+j}{\al+j}.
\end{equation}
It is easy to show that the functions $\{\phi^{i,j}_{n+\varepsilon}\}$ of the variable $n\in\Zset$ are linearly independent and satisfy
\begin{align*}
&\cJa\phi^{i,0}_{n+\varepsilon}=0,\text{ for }i=1,2\\
&\cJa\phi^{i,j}_{n+\varepsilon}=\phi^{i,j-1}_{n+\varepsilon},
\text{ for }j=1,\dots,k-1, \quad i=1,2.
\end{align*}
Thus, we can write the functions $\psi^{(j)}_n$ as a linear combination of 
$\phi^{i,j}_{n+\varepsilon}$ as follows
\begin{equation*}
\psi^{(j)}_n=\sum_{l=0}^{j}(\be_{j-l}\phi^{1,l}_{n+\varepsilon}+\ga_{j-l}\phi^{2,l}_{n+\varepsilon}).
\end{equation*}

If $\ga_0=0$ then $\psi^{(0)}_n=\be_0\neq 0$, i.e. we can take 
$\psi^{(0)}_n=1$. One can check that $\cJ_1$ in \eqref{3.1} coincides 
(up to a conjugation by a diagonal matrix) with $\cJ^{(\al-1)}$. This means that the operator 
$\cJh=\cJ_k$ can be obtained by a sequence of $(k-1)$ Darboux transformations 
starting from $\cJ^{(\al-1)}$. 
Thus we can assume that 
$\ga_0\neq 0$, hence we can take $\ga_0=1$. Since $\cQ$ depends only on the space
$\Span\{\psi^{(0)}_n,\psi^{(1)}_n,\dots, \psi^{(k-1)}_n\}$, and not on the choice
of the specific basis, we can can take $\ga_j=0$ for $j>0$. Thus we shall 
consider a basis for $\ker(\cQ)$ of the form
\begin{equation}\label{3.8}
\psi^{(j)}_n=\sum_{l=0}^{j}\be_{j-l}\phi^{1,l}_{n+\varepsilon}+\phi^{2,j}_{n+\varepsilon},
\end{equation}
depending on $\al\in\Nset$ and $k$ free parameters $\be_0,\be_1,\dots,\be_{k-1}$. We shall also 
normalize the matrix $\cQ$ by taking $g_n=1$ in \eqref{3.5}.

Now we consider the limit $\varepsilon\rightarrow 0$. Since $\lim_{\varepsilon\rightarrow 0}a_{-1+\varepsilon}=0$, it follows that the intertwining relation \eqref{3.2} holds 
for the semi-infinite parts of the bi-infinite matrices $\cJ_{0}$, $\cJh$ and $\cQ$. In other words, we have
\begin{equation}\label{3.9}
\Jh Q=Q \Ja,
\end{equation}
where  $\Ja$  is the semi-infinite matrix in \eqref{2.5}, $\Jh$ is a similar semi-infinite Jacobi matrix
$$\Jh=\left[\begin{matrix}
\bh_0&\ah_0\\
\ch_1&\bh_1&\ah_1\\
   &\ch_2&\bh_2&\ah_2\\
   &   &\ddots&\ddots&\ddots
\end{matrix}\right],$$ 
and $Q$ is a lower-triangular semi-infinite matrix, acting on functions $f_n$ by 
\begin{equation}\label{3.10}
Qf_n= \Wr_n(\psi^{(0)}_n,\psi^{(1)}_n,\dots,\psi^{(k-1)}_n,f_n),
\end{equation} 
where
\begin{equation}\label{3.11}
\psi^{(j)}_n=\sum_{l=0}^{j}\be_{j-l}\phi^{1,l}_{n}+\phi^{2,j}_{n},
\end{equation}
with the convention that $f_j=0$ when $j<0$.

One can easily write explicit formulas for the entries of the matrix $\Jh$. Indeed, if we consider the $k\times k$ determinants defined by
\begin{equation}\label{3.12}
\tau(n)=\Wr_n(\psi^{(0)}_n,\psi^{(1)}_n,\dots,\psi^{(k-1)}_n)
\end{equation}
and 
\begin{equation}\label{3.13}
\rho(n)=
\left\vert\begin{matrix}
\psi^{(0)}_n & \psi^{(1)}_n &\dots &\psi^{(k-1)}_n\\
\psi^{(0)}_{n-2} & \psi^{(1)}_{n-2} &\dots &\psi^{(k-1)}_{n-2}\\
\psi^{(0)}_{n-3} & \psi^{(1)}_{n-3} &\dots &\psi^{(k-1)}_{n-3}\\
\vdots                &                             &  &\vdots   \\

\psi^{(0)}_{n-k} & \psi^{(1)}_{n-k} &\dots &\psi^{(k-1)}_{n-k}
\end{matrix}\right\vert,
\end{equation}
then using equations \eqref{2.5}, \eqref{3.2} and \eqref{3.9} one can deduce that 
\begin{subequations}\label{3.14}
\begin{align}
\ah_n&=-\frac{\tau(n-1)}{\tau(n)}(n+1),\label{3.14a}\\
\bh_n&=2n+\al+1-\frac{\rho(n+1)}{\tau(n)}(n+1)+\frac{\rho(n)}{\tau(n-1)}n,\label{3.14b}\\
\ch_n&=-\frac{\tau(n)}{\tau(n-1)}(n+\al-k).\label{3.14c}
\end{align}
\end{subequations}
In particular, from \eqref{3.7} and the above formulas we see that the entries of the matrix $\Jh$ are well defined rational functions of $n$ when 
\begin{equation}\label{3.15}
\tau(n)\neq 0 \text{ for }n=-1,0,1,2,\dots.
\end{equation}
We shall call the set of parameters $\be=(\be_0,\be_1,\dots,\be_{k-1})$ {\em admissible} if \eqref{3.15} holds.  In the rest of the paper we shall work with admissible $\be$. One can show that
$$\tau(-1)=\frac{(-1)^{\binom{k}{2}}\be_0^k}{\prod_{j=1}^{k-1}(\al-j)^{k-j}},$$
which means that $\tau(-1)\neq0$ is equivalent to $\be_0\neq0$.

Finally, we 
denote by $\Lha_n(x)$ the polynomials defined by 
\begin{equation}\label{3.16}
\Lha_n(x)=Q \La_n(x)=
\Wr_n(\psi^{(0)}_n,\psi^{(1)}_n,\dots,\psi^{(k-1)}_n,\La_n(x)),
\end{equation}
which depend on the free parameters $\al\in\Nset$,  and 
$\be=(\be_0,\be_1,\dots,\be_{k-1})$. 

From \eqref{2.4}, \eqref{3.9}, \eqref{3.10} and \eqref{3.16} it follows 
that 
\begin{equation}\label{3.17}
\Jh \Lha_n(x)=x\Lha_n(x).
\end{equation}
From equations \eqref{3.14a} and \eqref{3.14c} we see that the off-diagonal entries $\ah_n$ and $\ch_n$ of the matrix $\Jh$ are 
nonzero. Therefore, by Favard's theorem, there exists a unique (up to a
multiplicative constant) moment functional $\cM$ for which 
$\{\Lha_n(x)\}_{n=0}^{\infty}$ is an orthogonal sequence, i.e.
\begin{equation}\label{3.18}
\cM(\Lha_n(x) \Lha_m(x))=0, \text{ for }n\neq m \text{ and }
\cM\left((\Lha_n(x))^2\right) \neq 0.
\end{equation}
More precisely, one can show that there exist constants 
$u_0,u_1,\dots,u_{k-1}$ 
such that the moment functional $\cM$ is given by the weight distribution 
\begin{equation}\label{3.19}
w(x)=\frac{1}{(\al-k)!}x^{\al-k}e^{-x}+\sum_{j=0}^{k-1}(-1)^ju_j\de^{(j)}(x),
\end{equation}
where $\de$ is the Dirac delta function. The parameters $\{u_j\}$ correspond to a 
different parametrization of the sequence of Darboux transformations \eqref{3.1}, see \cite[Theorem 2, p.~287]{GHH}. We shall continue to work with the parameters $\{\be_j\}$ and we give the explicit connection between $\{\be_j\}$ and $\{u_j\}$ in the examples when $k=1$ and $k=2$.

\begin{Remark}\label{re3.1}
In the proofs of the main results (\thref{th4.1} and \thref{th5.3}) we shall only use the fact that the functions $\psi^{(j)}_n$ are polynomials of $n$ (which follows immediately from equations \eqref{3.7} and \eqref{3.11}), but not their explicit form. Thus all statements will hold if we use \eqref{3.16} and arbitrary polynomials $\psi^{(j)}_n$ to define extensions of Laguerre polynomials. However, these polynomials will not (in general) satisfy the three-term recurrence relation \eqref{3.17} (hence they will not be orthogonal with respect to a moment functional). One can show that these polynomials will satisfy a higher-order recurrence relation. Indeed, note that if $f_n$ is a polynomial in $n$ of degree $s$, then $[\cJa]^{s+1}f_{n+\varepsilon}=0$. Using this fact we see that if we define $Q$ by \eqref{3.10} with polynomial functions $\psi^{(j)}_n$, then there exist $s\in\Nset$ and a semi-infinite $(2s+1)$-band matrix $\Jh$ which is zero everywhere, except for $s$ consecutive subdiagonals on either side of the main diagonal such that 
\begin{equation*}
\Jh Q=Q [\Ja]^{s}.
\end{equation*}
The above relation shows that the corresponding polynomials defined by \eqref{3.16} will satisfy a higher-order recurrence relation analogous to \eqref{3.17}. Interesting examples of such polynomials, satisfying a Sobolev-type orthogonality were studied by numerous authors, see for instance \cite{DM} and the references therein. The higher-order recurrence relations can naturally be connected to theory of matrix-valued orthogonal polynomials \cite{DurG} and the Toda lattice \cite{AvM}. We discuss the commutative algebra of differential operators for a specific family of Sobolev-type polynomials in the last section which extends the polynomials considered in \cite{KKB}.
\end{Remark}

\section{The commutative algebras $\cA$ and $\cD$}\label{se4}

In this section we prove that the polynomials $\Lha_n(x)$ defined in equation \eqref{3.16} are eigenfunctions 
for all operators in a commutative algebra of differential operators.

Let $\fD$ denote the associative algebra of ordinary differential operators generated by the operators $D_1=\dx$ and $D_2=x\dxx-x\dx$, i.e.
$$\fD=\Cset\left\langle \dx\,,\, x\dxx-x\dx \right\rangle.$$
Thus $\fD$ is an associative algebra generated by two elements satisfying the commutativity relation $[D_1,D_2]=D_1^2-D_1$.

Note that equation \eqref{2.2b} can be rewritten as 
\begin{equation}\label{4.1}
B \La_n(x)=n\La_n(x), \text{ where }B=-x\dxx-(\al+1-x)\dx\in\fD.
\end{equation}

Recall that $\tau(n)$ defined by \eqref{3.12} is a polynomial of $n$. We denote by $\cA$ the algebra of all polynomials $h(n)$ such that 
$h(n)-h(n-1)$ is divisible by $\tau(n-1)$ in $\Cset[n]$:
\begin{equation}\label{4.2}
\cA=\left\{h(n)\in\Cset[n]:\frac{h(n)-h(n-1)}{\tau(n-1)} \in\Cset[n]\right\}.
\end{equation}
It is not hard to see that $\cA$ contains a polynomial of every degree 
greater than $\deg(\tau)$. We can define the elements of $\cA$ as ``discrete" integrals of polynomials divisible 
by $\tau(n)$. To make this more precise we introduce the following notation: 
for $m,n\in\Zset$ and for a function $f(s)$ defined on $\Zset$ it will 
be convenient to use the following notation
\begin{equation*}
\int_{m}^{n}f(s)\dms=\begin{cases}
\sum_{s=m+1}^{n}f(s) &\text{if }n>m\\
0 &\text{if }n=m\\
-\sum_{s=n+1}^{m}f(s)&\text{if }n<m.
\end{cases}
\end{equation*}
Thus 
\begin{equation*}
\int_m^{n}f(s)\dms=f(n)+\int_m^{n-1}f(s)\dms\text{ for all }m, n\in\Zset.
\end{equation*}
Note that if $h\in\cA$ then 
\begin{equation}\label{4.3}
g(n)=\frac{h(n)-h(n-1)}{\tau(n-1)} \in\Cset[n]
\end{equation}
and 
\begin{equation}\label{4.4}
h(n)=\int_{0}^{n}g(s)\tau(s-1)\dms+h(0)=\int_{-1}^{n-1}g(s+1)\tau(s)\dms+h(0).
\end{equation}
With the above notations we can formulate the main result in this section.
\begin{Theorem}\label{th4.1}
For every $h\in\cA$ there exists $\Bh_h\in\fD$ such that
\begin{equation}\label{4.5}
\Bh_h\Lha_n(x)=h(n)\Lha_n(x), 
\end{equation}
for all $n\in\Nset_0$. Thus, $\cD=\{\Bh_h:h\in \cA\}$ is a commutative subalgebra of $\fD$, isomorphic 
to $\cA$.
\end{Theorem}

For the proof of \thref{th4.1} we shall need two lemmas. 
First we formulate a discrete analog of a lemma due to Reach \cite{R}. 
\begin{Lemma}\label{le4.2}
Let $f^{(0)}_{n}, f^{(1)}_{n},\dots,f^{(k+1)}_{n}$ be functions of a discrete 
variable $n$. Fix $n_1,n_2,\dots,n_{k+1}\in\Zset$ and let 
\begin{equation}\label{4.6}
F_n=\sum_{j=1}^{k+1}(-1)^{k+1+j}f^{(j)}_n\int_{n_j}^{n}f^{(0)}_{s}
\Wr_s(f^{(1)}_{s},\dots,\fh^{(j)}_{s},
\dots,f^{(k+1)}_{s})\dms,
\end{equation}
with the usual convention that the terms with hats are omitted. 
Then 
\begin{equation}\label{4.7}
\begin{split}
\Wr_n(f^{(1)}_{n},\dots,f^{(k)}_{n},F_n)=&
\int_{n_{k+1}}^{n-1}f^{(0)}_{s}\Wr_s(f^{(1)}_{s},\dots,f^{(k)}_{s})\dms\\
&\times \Wr_n(f^{(1)}_{n},\dots,f^{(k+1)}_{n}).
\end{split}
\end{equation}
\end{Lemma}

The above lemma was used in \cite{I1} to construct commutative algebras of differential operators for specific soliton solutions of the Toda lattice hierarchy. Since the proof is simple and we need it for one of the main results in the paper, we briefly sketch it below. 
\begin{proof}[Proof of \leref{le4.2}]
Note that
\begin{equation*}
\left|\begin{matrix}
f^{(1)}_{n}&f^{(2)}_{n}&\dots &f^{(k+1)}_{n}\\
f^{(1)}_{n-1}&f^{(2)}_{n-1}&\dots &f^{(k+1)}_{n-1}\\
\vdots & \vdots& &\vdots\\
f^{(1)}_{n-k+1}&f^{(2)}_{n-k+1}&\dots &f^{(k+1)}_{n-k+1}\\
f^{(1)}_{n-l}&f^{(2)}_{n-l}&\dots &f^{(k+1)}_{n-l}
\end{matrix}\right|=0, \text{ for every }l=0,1,\dots,k-1.
\end{equation*}
Expanding the above determinant along the last row we obtain
\begin{equation}\label{4.8}
\sum_{j=1}^{k+1}(-1)^{k+1+j}f^{(j)}_{n-l}
\Wr_n(f^{(1)}_{n},\dots\fh^{(j)}_{n},\dots,f^{(k+1)}_{n})=0
\text{ for }l=0,\dots,k-1.
\end{equation}
Using \eqref{4.7} and \eqref{4.8} we see that for $l=0,\dots,k-1$ we have
\begin{subequations}\label{4.9}
\begin{equation}\label{4.9a}
F_{n-l}=\sum_{j=1}^{k+1}(-1)^{k+1+j}f^{(j)}_{n-l}\int_{n_j}^{n}f^{(0)}_{s}
\Wr_s(f^{(1)}_{s},\dots,\fh^{(j)}_{s},
\dots,f^{(k+1)}_{s})\dms,
\end{equation}
and 
\begin{equation}\label{4.9b}
\begin{split}
F_{n-k}=&\sum_{j=1}^{k+1}(-1)^{k+1+j}f^{(j)}_{n-k}\int_{n_j}^{n}f^{(0)}_{s}
\Wr_s(f^{(1)}_{s},\dots,\fh^{(j)}_{s},
\dots,f^{(k+1)}_{s})\dms \\
&\qquad-f^{(0)}_{n}\Wr_n(f^{(1)}_{n},\dots,f^{(k+1)}_{n}).
\end{split}
\end{equation}
\end{subequations}
If we plug \eqref{4.9a} and \eqref{4.9b} in 
$\Wr_n(f^{(1)}_{n},\dots,f^{(k)}_{n},F_n)$, then most of the terms cancel by 
column elimination and we obtain \eqref{4.7}.
\end{proof}

\begin{Remark}\label{re4.3}
We list below important corollaries from the proof of \leref{le4.2}.

\item[(i)] Note that the right-hand side of \eqref{4.7} does not depend 
on the integers $n_1,n_2\dots,n_k$. Moreover, if we change $n_{k+1}$ then 
only the value of 
$$\int_{n_{k+1}}^{n-1}f^{(0)}_{s}\Wr_s(f^{(1)}_{s},\dots,f^{(k)}_{s})\dms$$
will change by an additive constant, which is independent of $n$ and 
$f^{(k+1)}_{n}$. Thus, instead of \eqref{4.6} we can write 
\begin{equation*}
F_n=\sum_{j=1}^{k+1}(-1)^{k+1+j}f^{(j)}_n\int^{n}f^{(0)}_{s}
\Wr_s(f^{(1)}_{s},\dots,\fh^{(j)}_{s},
\dots,f^{(k+1)}_{s})\dms,
\end{equation*}
leaving the lower bounds of the integrals (sums) blank and we can fix them 
at the end appropriately. This would allow us to easily change the variable, 
without keeping track of the lower end.
\item[(ii)] From \eqref{4.8} it follows that for every $l=-1, 0,1\dots,k-1$ we 
can write $F_n$ also as
\begin{equation*}
F_n=\sum_{j=1}^{k+1}(-1)^{k+1+j}f^{(j)}_n\int^{n+l}f^{(0)}_{s}
\Wr_s(f^{(1)}_{s},\dots,\fh^{(j)}_{s},
\dots,f^{(k+1)}_{s})\dms,
\end{equation*}
and changing the variable in the discrete integral we obtain
\begin{equation}\label{4.10}
F_n=\sum_{j=1}^{k+1}(-1)^{k+1+j}f^{(j)}_n\int^{n}f^{(0)}_{s+l}
\Wr_s(f^{(1)}_{s+l},\dots,\fh^{(j)}_{s+l},
\dots,f^{(k+1)}_{s+l})\dms.
\end{equation}
\item[(iii)] Let us apply (ii) with $l=\lfloor \frac{k-1}{2} \rfloor$ and consider the sum consisting of the first $k$ integrals:
\begin{equation*}
F_n^{(k)}=\sum_{j=1}^{k}(-1)^{k+1+j}f^{(j)}_n\int^{n}f^{(0)}_{s+l}
\Wr_s(f^{(1)}_{s+l},\dots,\fh^{(j)}_{s+l},\dots,f^{(k+1)}_{s+l})\dms.
\end{equation*}
Expanding each Wronskian determinant along the last column we can write 
$F_n^{(k)}$ as a sum of $k$ terms $F_n^{(k,m)}$, each one involving as 
integrand one of the functions $f^{(k+1)}_{s+m}$, where 
$m=l-k+1,l-k+2,\dots,l$.
We can use \eqref{4.8} once again, this time for the functions 
$f^{(1)}_{n},\dots,f^{(k)}_{n}$ (i.e. omitting $f^{(k+1)}_{n}$),
to change $s$ as follows:\\
\begin{itemize}
\item If $m\geq 0$ we can replace $n$ with $n-m$ in the upper limit of the 
integral, or equivalently, if we keep the upper limit of the integral to be 
$n$, we can replace $s$ by $s-m$ in the integrand. Thus $F_n^{(k,m)}$ will 
have $f^{(k+1)}_{s}$ 
as integrand (and the integration goes up to $n$).
\item If $m\leq -1$ we can replace $s$ by $s-m-1$, thus $F_n^{(k,m)}$ will have 
$f^{(k+1)}_{s-1}$ as integrand (and the integration goes up to $n$).
\end{itemize}
If $f^{(0)}_{n}, f^{(1)}_{n},\dots,f^{(k)}_{n}$ are polynomials of $n$, then summarizing the 
above observations we see that $F^{(k)}_{n}$ can be written as a sum of terms of the form
\begin{equation}\label{4.11}
r^{(0)}(n)\int^{n}[r^{(1)}(s)f^{(k+1)}_{s}+r^{(2)}(s)f^{(k+1)}_{s-1}] \dms,
\end{equation}
where $r^{(0)}(n)$, $r^{(1)}(n)$ and $r^{(2)}(n)$ are some polynomials of $n$.
\end{Remark}

\begin{Lemma}\label{le4.4}
For every $r^{(0)}(n)$, $r^{(1)}(n)$, $r^{(2)}(n) \in\Cset[n]$ there exists $\Bb \in\fD$ such that
\begin{equation}\label{4.12}
r^{(0)}(n)\,\int_{-1}^{n}[r^{(1)}(s)\La_{s}(x)+r^{(2)}(s)\La_{s-1}(x)]\dms= \Bb \La_n(x).
\end{equation}
\end{Lemma}
\begin{proof}
Using \eqref{4.1} and \eqref{2.3} we see that 
\begin{align*}
&\int_{-1}^{n}[r^{(1)}(s)\La_{s}(x)+r^{(2)}(s)\La_{s-1}(x)]\dms \\
&\qquad= \left[r^{(1)}(B)\left(1-\dx\right)-r^{(2)}(B+1)\dx\right] \int_{-1}^{n}\left(\La_{s}(x)-\La_{s-1}(x)\right)\dms\\
&\qquad= \left[r^{(1)}(B)\left(1-\dx\right)-r^{(2)}(B+1)\dx\right] \La_{n}(x).
\end{align*}
Using \eqref{4.1} once again we obtain \eqref{4.12} with 
$$\Bb=\left[r^{(1)}(B)\left(1-\dx\right)-r^{(2)}(B+1)\dx\right] r^{(0)}(B)\in\fD.$$
\end{proof}
We are now ready to present the proof of the main result in this section.
\begin{proof}[Proof of \thref{th4.1}]
Let $h\in\cA$ and let $g(n)\in\Cset[n]$ be the polynomial defined in \eqref{4.3}. We apply \leref{le4.2} with 
\begin{align*}
f^{(0)}_n&=g(n+1)\\
f^{(j)}_n&=\psi^{(j-1)}_n \text{ for }j=1,2,\dots,k\\
f^{(k+1)}_n&=\La_n(x).
\end{align*}
Using \eqref{3.12}, \eqref{3.16} and \eqref{4.4} we see that with the above choice of functions $f^{(j)}_n$, the right-hand side of equation of \eqref{4.7} is equal to $(h(n)+c)\Lha_n(x)$, where $c$ is a constant independent of $n$ and $x$. Thus, the proof of \eqref{4.5} will follow from \eqref{4.7} if we can show  
that, for an appropriate choice of the integers $n_j$, there exists a differential operator $\Bh\in\fD$ such that 
\begin{equation}\label{4.13}
F_n=\Bh\, \La_n(x).
\end{equation}
Indeed, if \eqref{4.13} holds, then we can pull $\Bh$ in front of the Wronskian determinant on the left-hand side of \eqref{4.7} which combined with 
\eqref{3.16} will establish \eqref{4.5}. We use now \reref{re4.3}. Note that the last term in the sum in equation \eqref{4.10} is a polynomial of $n$ times $\La_n(x)$ hence, using \eqref{4.1}, this term can be rewritten as $\Bh''\La_n(x)$ for some operator $\Bh''\in\fD$. Thus it remains to prove that 
for the sum $F^{(k)}_n$ consisting of the first $k$ terms we have
\begin{equation*}
F^{(k)}_n=\Bh'''\La_n(x) \text{ for some }\Bh'''\in\fD.
\end{equation*}
Since $F^{(k)}_n$ is a sum of terms of the form given in \eqref{4.11}, we see that the proof follows immediately from \leref{le4.4}.
\end{proof}

\begin{Remark}\label{re4.5}
Using the explicit formulas for $\Bh$ in \leref{le4.4} one can easily write exact formulas for the operator $\Bh_h$ in \eqref{4.5}. Moreover these formulas are relatively simple for small values of $k$. For instance, when $k=1$ we have
\begin{equation*}
\Bh_h=g(B+1)\left(\dx-1\right)\tau(B)+h(B+1),
\end{equation*}
where $g$ is defined by \eqref{4.3} and $B$ is the operator given in \eqref{4.1}. 
\end{Remark}

\begin{Remark}\label{re4.6}
It is perhaps useful to note that the order of operator $\Bh_h$ constructed in \thref{th4.1} equals twice the degree of the polynomial $h$. One way to see this is as follows. From equations \eqref{2.2} one can deduce that 
$$(n+\al)\La_{n-1}(x)=\left[-x\dxx-(\al+1)\dx\right]\La_{n}(x),$$
hence for every $j\in\Nset$ we have
$$(n+\al-j+1)_{j}\,\La_{n-j}(x)=\left[-x\dxx-(\al+1)\dx\right]^j\La_{n}(x).$$
This combined with definition \eqref{3.16} of $\Lha_n(x)$ shows that there exist a polynomial $\theta(n)$ 
(which is independent of $x$) and a differential operator $\Bt$ in $x$ (which is independent of $n$), such that 
\begin{equation}\label{4.14}
\theta(n)\Lha_{n}(x)=\Bt\La_n(x).
\end{equation}
Multiplying \eqref{4.5} by $\theta(n)$ and using equations \eqref{4.1} and \eqref{4.14} we deduce that
$$(\Bh_h \Bt-\Bt h(B))\La_n(x)=0 \text{ for all }n\in\Nset_0,$$
hence $\Bh_h \Bt=\Bt h(B)$. This shows that 
$$\mathrm{order}(\Bh_h)=\mathrm{order}(h(B))=2\deg_n(h(n)).$$
\end{Remark}

\begin{Remark}\label{re4.7}
Note that equation \eqref{2.1} can be used to define $\La_n(x)$ for all $n\in\Cset\setminus(-\Nset)$. Thus, we can think of $\La_n(x)$ as a 
meromorphic function of $n$ on $\Cset$ with simple poles at $-\Nset$. For $\al\in\Nset$, we can use \eqref{3.16} to obtain an extension of 
$\Lha_n(x)$ which is meromorphic in $\Cset$. By construction, these functions will satisfy the bi-infinite second-order difference equation
\begin{equation}\label{4.15}
\cJh \Lha_{n+\varepsilon}(x)=x\Lha_{n+\varepsilon}(x),
\end{equation}
where $\cJh$ is the bi-infinite matrix constructed in \seref{se3} and we interpret $\Lha_{n+\ep}(x)$ as a bi-infinite vector.
One can show that the function $\Lha_n(x)$ will also satisfy the differential equations \eqref{4.5}. Indeed, note that $\Lha_n(x)$ is an entire function of $x$. Expanding $\Lha_n(x)$ in a power series of $x$, we see that equation \eqref{4.5} will hold if and only if the coefficients of $x^k$ on both sides are equal for every $k\in\Nset_0$. Since these coefficients are polynomials of $n$, we see that \eqref{4.5} holds for all $n\in\Cset\setminus(-\Nset)$ if and only if it holds for all $n\in\Nset_0$. Thus, by \thref{th4.1}, equation \eqref{4.5} holds for every $n\in\Cset\setminus(-\Nset)$. This combined with 
equation \eqref{4.15} yields the bispectrality of the meromorphic (in $n$) function $\Lha_n(x)$.
\end{Remark}

\section{The commutative algebras $\cAb$ and $\cDb$}\label{se5}

Equations \eqref{3.17} and \eqref{4.5} establish the bispectral properties of the generalized Laguerre polynomials $\Lha_n(x)$. More generally, we can consider the algebra $\cAb$ of all complex valued functions $h(n)$ defined on $\Nset_0$ for which there exists a differential operator 
\begin{equation}\label{5.1}
\Bh_h=\sum_{j=0}^{m_h}b^{h}_j(x)\frac{d^j}{dx^j}
\end{equation} 
such that
\begin{equation}\label{5.2}
\Bh_h\Lha_n(x)=h(n)\Lha_n(x) \text{ for all }n\in\Nset_0.
\end{equation}
We denote by $\cDb$ the corresponding algebra of differential operators 
$$\cDb=\{\Bh_h:h\in\cAb\}.$$ 
Gr\"unbaum, Haine and Horozov proved that $\cAb$ contains polynomials $h(n)$ for every sufficiently large degree and observed that $\cA=\cAb$ in several examples, see  \cite[p.~295]{GHH}. \thref{th4.1} shows that $\cA\subset\cAb$, $\cD\subset\cDb$ and we prove below that $\cA=\cAb$,  $\cD=\cDb$ for generic set of parameters $\be=(\be_0,\be_1,\dots,\be_{k-1})$. 

First we list some simple properties of $\cAb$ and $\cDb$. If we consider equation \eqref{5.2} for $n=0,1,2,\dots,m_h$ we see that $b^{h}_j(x)$ is a polynomial of $x$ of degree at most $j$. Thus, $\cDb\subset\Cset\langle \dx ,x\dx\rangle$. If we compare the coefficients of $x^{n}$ on both sides of equation \eqref{5.2} we see that $h(n)$ is a polynomial in $n$ of degree at most $m_h$. In particular, this means that we can think of $\cAb$ as a subalgebra of $\Cset[n]$.

We denote by $\cR(h_1(n),h_2(n))$ the resultant of two polynomials $h_1(n)$ and $h_2(n)$. From \eqref{2.1} we see that 
\begin{equation}\label{5.3}
\frac{d^j \La_n}{dx^j}(0)=(-1)^{j}\binom{n+\al}{\al+j},
\end{equation}
which combined with \eqref{3.16} shows that
\begin{equation}\label{5.4}
\frac{d^j \Lha_n}{dx^j}(0)=\Wr_n\left(\psi^{(0)}_n,\psi^{(1)}_n,\dots,\psi^{(k-1)}_n, (-1)^{j}\binom{n+\al}{\al+j}\right).
\end{equation}
From equations  \eqref{3.7}, \eqref{3.11}, \eqref{3.12} and \eqref{5.4} we see that $\tau(n)$ and $\frac{d^j \Lha_n}{dx^j}(0)$ are polynomials of $n$ with coefficients which are polynomials of total degree at most $k$ in the free parameters $\be_0,\be_1,\dots,\be_{k-1}$. In particular, we can evaluate  $\tau(n)$ and $\frac{d^j \Lha_n}{dx^j}(0)$ for arbitrary $n\in\Cset$.

\begin{Definition}\label{de5.1} We say that an admissible set of parameters $\be=(\be_0,\be_1,\dots,\be_{k-1})$ is {\em generic} if 
\begin{equation}\label{5.5}
\fR(\be):=\cR\left(\tau(n),\Lha_n(0)\right)\neq0.
\end{equation}
\end{Definition}

Explicit formulas for $\fR(\be)$ in the cases $k=1$ and $k=2$ are given in the next section (see equation \eqref{6.1} for $k=1$ and \eqref{6.4} for $k=2$).

\begin{Remark}\label{re5.2}
From the observations above it follows that $\fR(\be)$ is a polynomial in the parameters $\be_0,\be_1,\dots,\be_{k-1}$. It is not hard to see that this polynomial is not identically equal to $0$. Indeed, let us assume that $\fR(\be)$ is identically equal to zero. Then there must exist $n_0\in\Cset$ which is a common root of $\tau(n)$ and $\Lha_{n}(0)$ for every admissible set of parameters $\be$. From \eqref{3.7}, \eqref{3.11} and \eqref{5.4} it follows easily that $\Lha_{n}(0)$ is independent of $\be_{k-1}$, hence $n_0$ must also be independent of $\be_{k-1}$. But then $n_0$ must be a root of $\frac{\partial \tau(n)}{\partial \be_{k-1}}$ which is independent of $\be_{k-2}$. This shows that $n_0$ must be a root of $\frac{\partial \Lha_{n}(0)}{\partial \be_{k-2}}$ which is independent of $\be_{k-3}$. Continuing this way (alternating $\tau(n)$ and $\Lha_{n}(0)$) we see that $n_0$ must be independent of $\be_0,\be_1,\dots,\be_{k-1}$. But this is impossible since for every fixed $n_0\in\Cset$, $\tau(n_0)$ is a polynomial of total degree $k$ in the parameters $\be_0,\be_1,\dots,\be_{k-1}$ which is not identically equal to $0$ (for instance, the coefficient of $\be_0^{k}$ is the nonzero constant $(-1)^{\binom{k}{2}}/{\prod_{j=1}^{k-1}(\al-j)^{k-j}}$).
\end{Remark}

The main result in this section is the following theorem.

\begin{Theorem}\label{th5.3}
For generic $\be=(\be_0,\be_1,\dots,\be_{k-1})$ we have $\cA=\cAb$ and $\cD=\cDb$.
\end{Theorem}
It is easy to see that when $k=1$, every admissible $\be_0$ is generic and therefore $\cA=\cAb$ and $\cD=\cDb$ (see \ssref{ss6.1}). However, one can show that in the very next simplest case ($\al=k=2$) we have $\cA\subsetneqq \cAb$ and $\cD\subsetneqq \cDb$ when $\be$ is not generic, see \ssref{ss6.2}.

For the proof of the above theorem we need two auxiliary facts.

\begin{Lemma}\label{le5.4}
Let $m\in\Nset$ and let $f^{(1)}(n),f^{(2)}(n),\dots,f^{(m+2)}(n)$ be functions of $n$. Then 
\begin{equation}\label{5.6}
\begin{split}
&\Wr_n(f^{(1)}(n),f^{(2)}(n),\dots,f^{(m+2)}(n))\,\Wr_n(f^{(2)}(n-1),f^{(3)}(n-1),\dots,f^{(m+1)}(n-1))\\
&=\Wr_n(f^{(2)}(n-1),f^{(3)}(n-1),\dots,f^{(m+2)}(n-1))\,\Wr_n(f^{(1)}(n),f^{(2)}(n),\dots,f^{(m+1)}(n))\\
&-\Wr_n(f^{(2)}(n),f^{(3)}(n),\dots,f^{(m+2)}(n))\,\Wr_n(f^{(1)}(n-1),f^{(2)}(n-1),\dots,f^{(m+1)}(n-1)).
\end{split}
\end{equation}
\end{Lemma}
\begin{proof}
The proof of \eqref{5.6} follows immediately by applying the Desnanot-Jacobi identity to the determinant $\Wr_n(f^{(1)}(n),f^{(2)}(n),\dots,f^{(m+2)}(n))$.
\end{proof}

\begin{Lemma}\label{le5.5}
Let $m\in\Nset$ and let $f^{(1)}(n),f^{(2)}(n),\dots,f^{(m+1)}(n)$ be polynomials of $n$. For $j=m$ and $j=m+1$, set $G^{(j)}(n)=\Wr_n(f^{(1)}(n),f^{(2)}(n),\dots,f^{(j)}(n))$. Then 
\begin{equation}\label{5.7}
\cR\left(G^{(m)}(n),G^{(m+1)}(n)\right)=\cR\left(G^{(m)}(n-1),G^{(m+1)}(n)\right).
\end{equation}
\end{Lemma}
\begin{proof}
We prove the statement by induction on $m$. When $m=1$ it is easy to see that 
\begin{equation*}
\begin{split}
&\cR\left(G^{(1)}(n),G^{(2)}(n)\right)=\cR\left(G^{(1)}(n-1),G^{(2)}(n)\right)\\
&\qquad=\cR\left(f^{(1)}(n-1),f^{(1)}(n)\right)\,\cR\left(f^{(1)}(n),f^{(2)}(n)\right).
\end{split}
\end{equation*}
Suppose now that the statement is true for some $m\in\Nset$ and we want to prove it for $m+1$. For $j=m+1$ and $j=m+2$ let us denote 
$$Q^{(j)}(n)=\Wr_n(f^{(2)}(n),f^{(3)}(n),\dots,f^{(j)}(n)).$$
By the induction hypothesis we have
\begin{equation}\label{5.8}
\cR\left(G^{(m+1)}(n),Q^{(m+1)}(n)\right)=\cR\left(G^{(m+1)}(n),Q^{(m+1)}(n-1)\right).
\end{equation}
From \leref{le5.4} it follows that 
\begin{subequations}\label{5.9}
\begin{align}
&\cR\left(G^{(m+1)}(n),G^{(m+2)}(n)\right)\,\cR\left(G^{(m+1)}(n),Q^{(m+1)}(n-1)\right)\nonumber\\
&\quad=\cR\left(G^{(m+1)}(n),-Q^{(m+2)}(n)G^{(m+1)}(n-1)\right)\nonumber\\
&\quad=\cR\left(G^{(m+1)}(n),Q^{(m+2)}(n)\right)\,\cR\left(G^{(m+1)}(n-1),G^{(m+1)}(n)\right),\label{5.9a}
\end{align}
and 
\begin{align}
&\cR\left(G^{(m+1)}(n-1),G^{(m+2)}(n)\right)\,\cR\left(G^{(m+1)}(n-1),Q^{(m+1)}(n-1)\right)\nonumber\\
&\quad=\cR\left(G^{(m+1)}(n-1),Q^{(m+2)}(n-1)G^{(m+1)}(n)\right)\nonumber\\
&\quad=\cR\left(G^{(m+1)}(n),Q^{(m+2)}(n)\right)\,\cR\left(G^{(m+1)}(n-1),G^{(m+1)}(n)\right).\label{5.9b}
\end{align}
\end{subequations}
The proof now follows from equations \eqref{5.8} and \eqref{5.9}.
\end{proof}

\begin{proof}[Proof of \thref{th5.3}]
We need to show that for generic $\be$, $\cAb\subset\cA$. Let $h$ be an arbitrary element of $\cAb$. By definition, there exists a differential operator $\Bh_h\in\cDb$ such that \eqref{5.2} holds. We know that $\Bh_h$ has polynomial (in $x$) coefficients and $h(n)\in\Cset[n]$. Evaluating \eqref{5.2} at $x=0$ and using the resulting equation for $n$ and $n-1$ we find
\begin{equation}\label{5.10}
\begin{split}
&\Lha_{n-1}(0)(\Bh_{h}\Lha_{n}(x))\Big\vert_{x=0}-\Lha_{n}(0)(\Bh_{h}\Lha_{n-1}(x))\Big\vert_{x=0}\\
&\qquad=\Lha_{n-1}(0)\Lha_{n}(0)(h(n)-h(n-1)).
\end{split}
\end{equation}
We look at both sides of \eqref{5.10} as polynomials of $n$ and we would like to prove next that $\tau(n-1)$ divides the left-hand side. For fixed $j\in\Nset$ we apply \leref{le5.4} with $m=k$ and $f^{(1)}(n)=\frac{d^j \La_n}{dx^j}(0)$, $f^{(2)}(n)=\psi^{(0)}_n$, $f^{(3)}(n)=\psi^{(1)}_n$, \dots, $f^{(k+1)}(n)=\psi^{(k-1)}_n$, $f^{(k+2)}(n)=\La_n(0)$ and we obtain:
\begin{equation*}
\begin{split}
&\Wr_n\left(\psi^{(0)}_n,\psi^{(1)}_n,\dots,\psi^{(k-1)}_n,\frac{d^j \La_n}{dx^j}(0),\La_n(0)\right)\tau(n-1)\\
&\qquad 
=\Lha_{n-1}(0)\frac{d^j \Lha_{n}}{dx^j}(0)-\Lha_{n}(0)\frac{d^j \Lha_{n-1}}{dx^j}(0).
\end{split}
\end{equation*}
The last equation shows that $\tau(n-1)$ divides $\Lha_{n-1}(0)\frac{d^j \Lha_{n}}{dx^j}(0)-\Lha_{n}(0)\frac{d^j \Lha_{n-1}}{dx^j}(0)$ for every $j\in\Nset$ which proves that $\tau(n-1)$ divides the left-hand side of \eqref{5.10}. Therefore, $\tau(n-1)$ must divide 
 the right-hand side of \eqref{5.10}. It remains to show now that for generic $\be$ the polynomials $\tau(n-1)$ and $\Lha_{n-1}(0)\Lha_{n}(0)$ are relatively prime. From \leref{le5.5} it follows that 
 $$\cR\left(\tau(n),\Lha_{n}(0)\right)=\cR\left(\tau(n-1),\Lha_{n}(0)\right).$$
Thus using the notations in \deref{de5.1} we find
$$\cR\left(\tau(n-1),\Lha_{n-1}(0)\Lha_{n}(0)\right)=\left[\fR(\be)\right]^2\neq0 \text{ for generic }\be,$$
completing the proof.
\end{proof}

\section{Explicit examples}\label{se6}

\subsection{One Darboux step}\label{ss6.1}
Suppose that $\al\in\Nset$ and $k=1$. Then 
$$\tau(n)=\psi^{(0)}_n=\beta_0+\binom{n+\al}{\al}$$
and formula \eqref{3.16} gives
$$\Lh^{\al,(\be_0)}_n(x)=\Wr_n\left(\psi^{(0)}_n,\La_n(x)\right).$$
The polynomials $\Lh^{\al,(\be_0)}_n(x)$ are orthogonal on $[0,\infty)$ with respect to the weight distribution 
$$w(x)=\frac{1}{(\al-1)!}x^{\al-1}e^{-x}+u_0\de(x), \text{ where }u_0=\frac{1}{\be_0}.$$
From the explicit formula above it is easy to see that 
$$\Lh^{\al,(\be_0)}_n(0)=-\be_0\binom{n+\al-1}{\al-1}.$$
Combining the last formula with the formula for $\tau(n)$ and definition \deref{de5.1}  we see that 
\begin{equation}\label{6.1}
\fR(\be_0)=\cR\left(\tau(n),\Lh^{\al,(\be_0)}_n(0)\right)=(-1)^{\al}\frac{\be_0^{2\al-1}}{[(\al-1)!]^{\al}}.
\end{equation}
Recall that $\be_0\neq0$ since we work with an admissible set of parameters (i.e. $\tau(n-1)\neq 0$ for $n\in\Nset_0$) and therefore every admissible $\be_0$ is generic. This means that when $k=1$ we always have $\cA=\cAb$ and $\cD=\cDb$. We write below an explicit set of generators for both algebras.

For $j\in\Nset_0$ we consider the polynomials of $n$ defined by
\begin{equation}\label{6.2}
h^{(j)}(n)=\be_0\binom{n+\al+j}{j+1}+\binom{\al+j}{\al}\binom{n+\al+j}{\al+j+1}.
\end{equation}
Note that 
$$g^{(j)}(n):=\frac{h^{(j)}(n)-h^{(j)}(n-1)}{\tau(n-1)}=\binom{n+\al+j-1}{j}\in\Cset[n].$$
This shows that $h^{(j)}(n)\in\cA$ for all $j\in\Nset_0$. Moreover, since $\deg_n(h^{(j)}(n))=\al+j+1$ we see that 
$$\cA=\Cset[h^{(0)}(n),h^{(1)}(n),\dots,h^{(\al)}(n)].$$
Using \reref{re4.5} we can write explicit formulas for the operators $\Bh_j:=\Bh_{h^{(j)}}$:
\begin{equation*}
\Bh_j=g^{(j)}(B+1)\left(\dx-1\right)\tau(B)+h^{(j)}(B+1),
\end{equation*}
where $B$ is the second-order differential operator for Laguerre polynomials given in \eqref{4.1}. 
The algebra of differential operators $\cD$ is therefore generated by the operators $\Bh_0, \Bh_1,\dots,\Bh_{\al}$ of orders $2\al+2,2\al+4,\dots,4\al+2$. 
The construction of the operator $\Bh_0$ of minimal order $2\al+2$ has a long history. For $\al=1$, $\al=2$ and $\al=3$ the differential equation of minimal order 
goes back to the works of H.~L.~Krall, A.~M.~Krall and L.~L.~Littlejohn \cite{Kr,KL,Kr1,Kr2,L}. The explicit form of $\Bh_0$ for arbitrary $\al\in\Nset$ was obtained by J.~Koekoek and R.~Koekoek \cite{KK}.

\subsection{Two Darboux steps}\label{ss6.2}
Let us take now $k=2$ and let $\al$ be an integer greater than $1$. The polynomials $\Lha_n(x)=\Lh^{\al,(\be_0,\be_1)}_n(x)$ are defined by \eqref{3.16} with 
\begin{subequations}\label{6.3}
\begin{align}
\psi^{(0)}_n&=\beta_0\phi^{1,0}_{n}+\phi^{2,0}_{n}=\beta_0+\binom{n+\al}{\al}\label{6.3a}\\
\psi^{(1)}_n&=\beta_1\phi^{1,0}_{n}+\beta_0\phi^{1,1}_{n}+\phi^{2,1}_{n}=\beta_1+\frac{\be_0}{\al-1}(n+1)-\binom{n+\al+1}{\al+1}.\label{6.3b}
\end{align}
\end{subequations}
They are orthogonal on $[0,\infty)$ with respect to the weight distribution 
$$w(x)=\frac{1}{(\al-2)!}x^{\al-2}e^{-x}+u_0\de(x)-u_1\de'(x),$$ 
where 
$$u_0=-\frac{(\al-1)(\be_0+\be_1)}{\be_0^2} \text{ and }u_1=\frac{\al-1}{\be_0}.$$
For $\Lh^{\al,(\be_0,\be_1)}_n(0)$ we obtain the following formula:
\begin{equation*}
\Lh^{\al,(\be_0,\be_1)}_n(0)=-\frac{\be_0}{(\al-1)!}(n+1)_{\al-2}\left[\binom{n+\al-1}{\al}+\be_0\right].
\end{equation*}
Using the explicit formulas for $\tau(n)$ and $\Lh^{\al,(\be_0,\be_1)}_n(0)$ one can derive the following formula for the resultant $\fR(\be)$:
\begin{equation}\label{6.4}
\fR(\be_0,\be_1)=\frac{2^{\al}\be_0^{5\al-5}}{(\al-1)^{2\al-2}(\al+1)^{\al}[(\al-1)!]^{4\al}}
\left[ \be_0^{\al}\binom{\frac{\al-1}{2}-\frac{\be_1(\al^2-1)}{2\be_0\al}}{\al}+\be_0^{\al+1}\right].
\end{equation}
Ignoring the unessential constant factors and $\be_0^{5\al-5}$ (which is nonzero), we see that admissible $\be_0$ and $\be_1$ are generic if and only if they do not satisfy the algebraic equation
\begin{equation}\label{6.5}
(2\al)^{\al}\al!\be_0^{\al+1}+\prod_{j=0}^{\al-1}\left[(\al^2-\al-2\al j)\be_0-(\al^2-1)\be_1\right]=0.
\end{equation}
Using the explicit connection between $(u_0,u_1)$ and $(\be_0,\be_1)$, one can easily rewrite the above equation in terms of $u_0$ and $u_1$.

Note that $\tau(n)$ is a polynomial in $n$ of degree $2\al$. Thus, 
the commutative algebra $\cA$ is generated by polynomials of degrees $2\al+1,2\al+2,\dots,4\al+1$. The corresponding algebra of differential operators $\cD$ is therefore generated by operators of orders $4\al+2,4\al+4,\dots,8\al+2$. 

Let us consider in details the simplest case $\al=2$. The weight distribution  is
$$w(x)=e^{-x}+u_0\de(x)-u_1\de'(x).$$ 
Equation \eqref{6.5} becomes 
\begin{equation}\label{6.6}
32\be_0^3-4\be_0^2+9\be_1^2=0,
\end{equation}
or equivalently, in terms of $u_0=-(\be_0+\be_1)/\be_0^2$ and $u_1=1/\be_0$ we get
$$9u_0^2+18u_0u_1+5u_1^2+32u_1=0.$$
If $\be_0$ and $\be_1$ do not satisfy \eqref{6.6}, then $\cAb=\cA$ is an algebra generated by polynomials of degrees $5,6,\dots,9$ and $\cDb=\cD$ is an algebra of differential operators generated by  operators of orders $10,11,\dots,18$. The operator of minimal order 10 is given explicitly in \cite[pages 292--295]{GHH} (in slightly different notations). 

However, if $\be_0$ and $\be_1$ satisfy \eqref{6.6}, then one can show that 
$\cAb$ contains the following polynomial of degree $4$ 
$$h(n)=n^4+2n^3+(28\be_0-1)n^2+(36\be_1+28\be_0-2)n.$$
Thus if \eqref{6.6} holds we have 
$\cA\subsetneqq \cAb$, $\cD\subsetneqq \cDb$ and the polynomials $\Lh^{\al,(\be_0,\be_1)}_n(x)$ satisfy a differential equation of order 8.

\subsection{Sobolev-type orthogonality}\label{ss6.3}
Let $\al$ be an integer greater than $1$ and let us consider a two-dimensional subspace of the space $\Span\{\phi^{1,0}_n,\phi^{2,0}_n,\phi^{1,1}_n,\phi^{2,1}_n\}$ with basis $\{\psi^{(0)}_n,\psi^{(1)}_n\}$. If we assume that the coefficient of $\phi^{2,0}_n$ in $\psi^{(0)}_n$ and the coefficient of $\phi^{2,1}_n$ in $\psi^{(1)}_n$  are nonzero, we can pick $\psi^{(0)}_n$ and $\psi^{(1)}_n$ as follows:
\begin{subequations}\label{6.7}
\begin{align}
\psi^{(0)}_n&=\beta_0\phi^{1,0}_{n}+l_0\phi^{1,1}_{n}+\phi^{2,0}_{n}=\beta_0+\frac{l_0}{\al-1}(n+1)+\binom{n+\al}{\al}\label{6.7a}\\
\psi^{(1)}_n&=\beta_1\phi^{1,0}_{n}+l_1\phi^{1,1}_{n}+\phi^{2,1}_{n}=\beta_1+\frac{l_1}{\al-1}(n+1)-\binom{n+\al+1}{\al+1}.\label{6.7b}
\end{align}
\end{subequations}
If we take $l_0=0$ and $l_1=\be_0$ equations \eqref{6.7} reduce to \eqref{6.3}. If we define a bi-infinite matrix $\cQ$ with basis $\{\psi^{(0)}_{n+\varepsilon},\psi^{(1)}_{n+\varepsilon}\}$ then using the notations in \seref{se3} we have
\begin{equation*}
\ker(\cQ)\subset\ker[\cJ^{(\al)}]^2,
\end{equation*}
but for arbitrary $l_0$ and $l_1$
$$\cJ^{(\al)}(\ker(\cQ))\not\subset\ker(\cQ).$$
Since only the first part of equation \eqref{3.4} holds, we cannot expect to find a tridiagonal matrix $\cJh$, which is obtained by a sequence of two Darboux transformations from $\cJa$, but we can find a pentadiagonal matrix $\cJh$ which is a Darboux transformation of $[\cJ^{(\al)}]^2$ as follows:
$$[\cJ^{(\al)}]^2=\cP\cQ\curvearrowright \cJh=\cQ\cP.$$
In the semi-infinite limit $(\varepsilon\rightarrow 0)$, this construction leads to a semi-infinite upper-triangular matrix $P$ and a semi-infinite lower-triangular matrix $Q$, with two diagonals above (resp. below) the main diagonal:
\begin{equation*}
P=\left[\begin{matrix}
p_{0,0}&p_{0,1}&p_{0,2} & \\
&p_{1,1}&p_{1,2}&p_{1,3} & \\
&&p_{2,2}&p_{2,3}&p_{2,4} & \\
 &  &   &\ddots&\ddots&\ddots
\end{matrix}\right], \qquad 
Q=\left[\begin{matrix}
q_{0,0}\\
q_{1,0} &q_{1,1}\\
q_{2,0} & q_{2,1} & q_{2,2} \\
& \ddots &\ddots &\ddots 
\end{matrix}\right],
\end{equation*}
satisfying 
\begin{equation}\label{6.8}
[\Ja]^2=PQ \text{ and }\Jh=QP,
\end{equation}
where $\Jh$ is a semi-infinite pentadiagonal matrix. 

Let us consider the polynomials $\Lh^{\al,(\be_0,\be_1,l_0,l_1)}_n(x)=\Lh_n(x)$ obtained by applying $Q$ to $\La_n(x)$, or equivalently by using the Wronskian formula \eqref{3.16} with $k=2$ and $\psi^{(0)}_n,\psi^{(1)}_n$ given in \eqref{6.7}, i.e. 
\begin{equation}\label{6.9}
\Lh_n(x)=Q\La_n(x)=\Wr_n(\psi^{(0)}_n,\psi^{(1)}_n,\La_n(x)).
\end{equation}
As we noted in \reref{re3.1},  \thref{th4.1}  and \thref{th5.3} can be applied for $\Lh_n(x)$ (since the proof used only the fact that $\psi^{(j)}_n$ are polynomials of $n$). Thus, if we define $\tau(n)$ and $\cA$ by equations \eqref{3.12} and \eqref{4.2}, then for every $h\in\cA$ there exists $\Bh_h\in\fD$, such that
\begin{equation*}
\Bh_h\Lh_n(x)=h(n)\Lh_n(x).
\end{equation*}
We can also consider the algebras $\cAb$, $\cDb$ defined at the beginning of \seref{se5} and for generic parameters $\be_0,\be_1,l_0,l_1$ we will have $\cA=\cAb$ and $\cD=\cDb$.

The polynomials $\Lh_n(x)$ will not (in general) be orthogonal with respect to a moment functional. However, if we pick ``carefully" the free parameters 
$\be_0,\be_1,l_0,l_1$ the polynomials $\Lh_n(x)$ will be orthogonal with respect to the Sobolev inner product
\begin{subequations}\label{6.10}
\begin{equation}\label{6.10a}
\langle F(x), G(x)\rangle=\frac{1}{(\al-2)!}\int_0^{\infty}F(x)G(x)x^{\al-2}e^{-x}dx+[F(0),F'(0)]\,A\,\left[\begin{matrix} G(0)\\ G'(0)\end{matrix}\right],
\end{equation}
where $A$ is a symmetric $2\times 2$ matrix
\begin{equation}\label{6.10b}
A=\left[\begin{matrix} u_0 & u_1 \\ u_1 & v_0\end{matrix}\right].
\end{equation}
\end{subequations}
Indeed, note that equation \eqref{6.9} gives 
\begin{equation*}
\Lh_n(x)=q_{n,n}\La_{n}(x)+q_{n,n-1}\La_{n-1}(x)+q_{n,n-2}\La_{n-2}(x),
\end{equation*}
which combined with the orthogonality relations for $\La_n(x)$ shows that for the inner product defined in \eqref{6.10} we have
\begin{equation*}
\langle \Lh_n(x), x^j\rangle =0\text{ when } 2\leq j <n.
\end{equation*}
Thus, we need to pick the parameters so that $\langle \Lh_n(x), x^j\rangle =0$ when $j\in\{0,1\}$ and $n>j$. From equations \eqref{6.8}, \eqref{6.9} and \eqref{2.4} we see that $P\Lh_n(x)=x^2\La_n(x)$, or equivalently
\begin{equation*}
p_{n,n}\Lh_{n}(x)+p_{n,n+1}\Lh_{n+1}(x)+p_{n,n+2}\Lh_{n+2}(x)=x^2\La_n(x).
\end{equation*}
It is easy to see that the numbers $p_{n,n+2}$ are nonzero, and therefore the last equation shows by induction that
\begin{itemize}
\item if $\langle \Lh_n(x), 1\rangle =0$ for $n=1$ and $n=2$, then $\langle \Lh_n(x), 1\rangle =0$ for all $n\in\Nset$.
\item if $\langle \Lh_n(x), x\rangle =0$ for $n=2$ and $n=3$, then $\langle \Lh_n(x), x\rangle =0$ for all $n\geq 2$.
\end{itemize}
Summarizing the above observations we see that the polynomials $\Lh_n(x)$ will be mutually orthogonal with respect to the inner product \eqref{6.10} if and only if $\langle \Lh_n(x), 1\rangle =0$ for $n\in\{1,2\}$ and  $\langle \Lh_n(x), x\rangle =0$ for $n\in\{2,3\}$. If $\det(A)\neq0$, one can show that these four equations are satisfied if we pick $\be_0,\be_1,l_0,l_1$ as follows
\begin{subequations}\label{6.11}
\begin{align}
&\be_0=-\frac{(\al-1)u_1}{\det(A)},&& \be_1=\frac{(\al-1)(u_0+u_1)}{\det(A)},\label{6.11a}\\
&l_0=\frac{(\al-1)v_0}{\det(A)},&& l_1=-\frac{(\al-1)(u_1+v_0)}{\det(A)}.\label{6.11b}
\end{align}
\end{subequations}
If $v_0=0$, then $l_0=0$, $l_1=\be_0$ and we obtain the polynomials in the previous example (iteration of two Darboux transformations). For $v_0\neq 0$, the polynomials are no longer associated with a moment functional (since $\langle x,x \rangle \neq \langle x^2,1 \rangle$).

Since $\deg_n(\tau(n))=2\al$ the commutative algebra $\cA$ is generated by polynomials of degrees $2\al+1,2\al+2,\dots,4\al+1$. In the case when the matrix $A$ is diagonal (i.e. $u_1=0$) the operator of minimal order $4\al+2$ was computed explicitly in \cite{KKB}.

Finally, if the matrix $A$ is singular, we can still apply the above construction by a limiting procedure, but we need to pick a different basis for the two-dimensional subspace of $\Span\{\phi^{1,0}_n,\phi^{2,0}_n,\phi^{1,1}_n,\phi^{2,1}_n\}$. For instance, if we want to consider the inner product \eqref{6.10a} with 
\begin{equation*}
A=\left[\begin{matrix} 0 & 0 \\ 0 & v_0\end{matrix}\right],
\end{equation*}
then the polynomials $\Lh_n(x)$ will be given by \eqref{6.9} with $\psi^{(0)}_n$ and $\psi^{(1)}_n$ defined by
\begin{subequations}\label{6.12}
\begin{align}
\psi^{(0)}_n&=n+1\label{6.12a}\\
\psi^{(1)}_n&=\binom{n+\al+1}{\al+1}-\binom{n+\al}{\al}-\frac{\al-1}{v_0}.\label{6.12b}
\end{align}
\end{subequations}
Note that in this case $\deg_n(\tau(n))=\al+1$, which means that our theory will produce a differential operator in $\cD$ of minimal order $2\al+4$, in agreement with the results in \cite{KKB}.

\end{document}